\documentclass[journal]{IEEEtran}
\usepackage[utf8]{inputenc}
\usepackage{cite}
\usepackage{amsmath,amssymb,amsfonts,amsthm}
\usepackage{mathrsfs}
\usepackage{float}
\usepackage{graphicx}
\usepackage{subcaption}
\usepackage{graphicx}
\newtheorem{thm}{Theorem}
\newtheorem{lem}{Lemma}

\title{\LARGE \bf
Robust Boundary Stabilization of Stochastic Hyperbolic PDEs
}

\author{Yihuai Zhang, Jean Auriol, Huan Yu$^{*}$
\thanks{Yihuai Zhang and Huan Yu are with the Thrust of Intelligent Transportation, The Hong Kong University of Science and Technology (Guangzhou), Guangzhou, 511400, China. Huan Yu is also affiliated with the Department of Civil and Environmental Engineering, The Hong Kong University of Science and Technology, Hong Kong SAR, China.}
\thanks{Jean Auriol is with the Universit\'{e} Paris-Saclay, CNRS, CentraleSup\'{e}lec, Laboratoire des Signaux et Syst\`{e}mes, 91190, Gif-sur-Yvette, France}
\thanks{$^{*}$ Corresponding author: Huan Yu ({\tt\small huanyu@ust.hk)}}}

\begin{document}

\maketitle
\thispagestyle{empty}
\pagestyle{empty}
\bibliographystyle{abbrv}

\begin{abstract}

This paper proposes a backstepping boundary control design for robust stabilization of linear first-order coupled hyperbolic partial differential equations (PDEs) with Markov-jumping parameters. The PDE system consists of 4 $\times$ 4 coupled hyperbolic PDEs whose first three characteristic speeds are positive and the last one is negative. We first design a full-state feedback boundary control law for a nominal, deterministic system using the backstepping method. Then, by applying Lyapunov analysis methods, we prove that the nominal backstepping control law can stabilize the PDE system with Markov jumping parameters if the nominal parameters are sufficiently close to the stochastic ones on average. The mean-square exponential stability conditions are theoretically derived and then validated via numerical simulations.
\end{abstract}

\section{Introduction}

Hyperbolic partial differential equations (PDEs) find applications in many engineering areas, including transportation systems~\cite{yu2022traffic}, open-channel flows \cite{de2003boundary}. Extensive research efforts have been made for boundary control problems of hyperbolic PDE systems. 
PDE backstepping, as a feedback stabilization method, has been developed for general first-order coupled hyperbolic PDEs in a series of work~\cite{coron2013local,di2013stabilization,hu2015control}. Further advances generalize the PDE backstepping for adaptive control~\cite{anfinsen2018adaptive}, robust stabilization to delay and disturbances~\cite{auriol2018delay,lamare2018robust}. These research results mainly pertain to deterministic PDE systems. On the other hand, scarce studies focus on stochastic PDE systems, the stochasticity being usually caused by uncertain parameters in real applications. The question of boundary control of uncertain hyperbolic PDEs with Markov jump parameters via the backstepping method remains open.

Previous results have focused on stochastic stability analysis with distributed or boundary controllers using linear matrix inequalities~\cite{amin2011exponential,bolzern2006almost,do2012continuous}.
Prieur~\cite{prieur2014stability} modeled the abrupt changes of boundary conditions as a piecewise constant function and derived sufficient conditions for the exponential stability of the switching system. Wang et al.~\cite{wang2012stochastically} examined the robustly stochastically exponential stability and stabilization of uncertain linear first-order hyperbolic PDEs with Markov jumping parameters, deriving sufficient stability conditions using linear matrix inequalities (LMIs) based on integral-type stochastic Lyapunov functional.  Zhang~\cite{zhang2017stochastic} studied traffic flow control of Markov jump hyperbolic systems, employing LMIs to derive sufficient conditions for boundary exponential stability. Auriol~\cite{auriol2023mean} first considered the mean-square exponential stability of a $2\times 2$ stochastic hyperbolic system using backstepping.

The main results of this paper are that we propose a robust stabilizing backstepping control law for a  $4\times 4$ stochastic hyperbolic PDEs. We first design a backstepping boundary control law to stabilize the nominal system without Markov jumping parameters. We prove that the nominal control law can still stabilize the stochastic PDE system, provided the nominal parameters are sufficiently close to the stochastic ones on average. The stability conditions are derived using a Lyapunov analysis. The contribution of this paper extends to both theoretical advancements and practical applications. 

This paper is organized as follows: In Section II, we introduce the stochastic hyperbolic PDE system and state the problem under consideration. In Section III, the nominal boundary controller is designed, and the mean-square exponential stabilization of the stochastic PDEs is proposed. In Section IV, a Lyapunov analysis is conducted to prove the nominal control law achieves the mean-square exponential stability of the PDE system with Markov jumping parameters. In Section V numerical simulations verify the theoretical results.

\section{Problem Statement}
In this paper, we consider the following stochastic PDE system 
\begin{align}
        \mathbf{w}_t^+(x,t)+\Lambda^{+}(t) \mathbf{w}^+_x(x,t) =& \Sigma^{++}(x,t)\mathbf{w}^+(x,t) \nonumber\\
        &+\Sigma^{+-}(x,t) \mathbf{w}^-(x, t) ,\label{clpsys1}\\
        \mathbf{w}^-_t(x, t)-\Lambda^{-}(t) \mathbf{w}^-_x(x, t) =& \Sigma^{-+}(x,t)\mathbf{w}^+,\label{clpsys2}
\end{align}
with the boundary conditions 
\begin{align}
    &{\mathbf{w}^+(0,t)}  = {Q}(t) \mathbf{w}^-(0, t), \label{clpsys3} \\
    &\mathbf{w}^-(1, t) = {R}(t)\mathbf{w}^+(1,t)+{U}(t), \label{clpsys4}
\end{align}
where $\mathbf{w}^{+}= [w_1, w_2, w_3]^\mathsf{T}$, $\mathbf{w}^{-} = w_4$, where the spatial and time domain are defined in $(x,t) \in [0,1] \times \mathbb{R}^+$. The different transport matrices are defined as
\begin{align*}
\Lambda^{+}(t) = \begin{bmatrix}
    \lambda_{1}(t) & 0 & 0 \\
0 & \lambda_{2}(t) & 0 \\
0 & 0 & \lambda_{3}(t)
\end{bmatrix},\Lambda^{-}_i =\lambda_{4}(t)
\end{align*}
The coefficient matrices are $\Sigma^{++}(x,t)\in \mathbb{R}^{3\times 3}$, $\Sigma^{+-}(x,t)\in \mathbb{R}^{1\times 3}$, $\Sigma^{-+}(x,t)\in \mathbb{R}^{3\times 1}$, ${Q}(t)\in \mathbb{R}^{3\times 1}$, ${R}(t)\in\mathbb{R}^{1\times 3}$. The function $U(t)$ is the control input. All the parameters of the systems are stochastic and we denote their concatenation $\mathscr{X}(t) = \{\Lambda^+(t),\Lambda^-(t),\Sigma^{++}(x,t),\Sigma^{+-}(x,t),\Sigma^{-+}(x,t), Q(t), R(t)\}$. We consider that the set $\mathscr{X}(t)$ corresponds to a homogeneous continuous Markov process  $\mathscr{X}(t), t \in \mathbb{R}^+$ with a finite number of states $ \mathscr{S} = \{\mathscr{X}_1,\mathscr{X}_2,\dots,\mathscr{X}_r\}$, whose realization is right continuous. For instance, we have $\mathscr{X}_j=\{\Lambda_j^+,\Lambda_j^-,\Sigma_j^{++}(x),\Sigma_j^{+-}(x),\Sigma_j^{-+}(x), Q_j, R_j\}$.

The transition probabilities $P_{ij}(t_1,t_2)$ denote the probability to switch from mode $\mathscr{X}_i$ at time $t_1$ to mode $\mathscr{X}_j$ at time $t_2$ ($(i,j) \in \{1, \dots, r\}^2,~0 \leq t_1 \leq t_2$). They satisfy ~$P_{ij}: \mathbb{R}^2 \rightarrow [0,1]$ with $\sum_{j=1}^{r} P_{ij}(t_1,t_2)=1$. Moreover, for $\varrho <t$, the transition probabilities ${P}_{ij}$ follows the Kolmogorov equation~\cite{hoyland2009system,kolmanovsky2001mean},
\begin{align}
    &\partial_t  {P}_{i j}(\varrho, t)=-\mathfrak{c}_j(t)  {P}_{i j}(\varrho, t)+\sum_{k=1}^{r}  {P}_{i k}(\varrho, t) \tau_{k j}(t),\notag\\ 
    &  {P}_{i i}(\varrho, \varrho)=1, \text { and }  {P}_{i j}(\varrho, \varrho)=0 \text { for } i \neq j
    \label{koleq}
\end{align}
where the $\tau_{ij}$ and $\mathfrak{c}_j=\sum_{k=1,k\neq j}^{r} \tau_{j k}$ are non-negative-valued functions such that $\tau_{ii}(t)=0$. The functions $\tau_{ij}$ are upper bounded by a constant $\tau^\star$.  For $\mathscr{X}_j \in \mathscr{S}$, we denote $||\mathscr{X}_j(t)||$ as
\begin{align}
||\mathscr{X}_j(t)||&=(||\Lambda^+_j(t)||^2+||\Lambda^-_j(t)||^2+||Q_j(t)||^2+||R_j(t)||^2\nonumber \\
&+\sup_{x\in [0,1]}||\Sigma_j^{++}(x)||^2+\sup_{x\in [0,1]}||\Sigma_j^{+-}(x)||^2\nonumber\\
&+\sup_{x\in [0,1]}||\Sigma_j^{-+}(x)||^2)^{\frac{1}{2}},
\end{align}
where we have used the standard Euclidean norm.
We assume that there exists known lower bounds $\underline{\mathscr{X}}$ and upper bounds $\overline{\mathscr{X}}$ such that for all $j$ , $\underline{\mathscr{X}} \leq ||\mathscr{X}_j||\leq \overline{\mathscr{X}}$. Moreover, we assume that the lower bounds of the stochastic velocities are always positive. More precisely, we have  $\underline{\Lambda_i^+}>0$, $\underline{\Lambda_i^-}>0$, which implies $\underline{\lambda_{1i}}>0$, $\underline{\lambda_{2i}}>0$, $\underline{\lambda_{3i}}>0$ and $\underline{\lambda_{4i}}>0$. Using the notations $\mathbf{w} = [\mathbf{w}^{+}, \mathbf{w}^{-}]$ and $\Lambda_i = \text{Diag}\{\lambda_{1}^i,\lambda_{2}^i,\lambda_{3}^i,-\lambda_{4}^i\}$, the system \eqref{clpsys1}-\eqref{clpsys4} can be rewritten in the compact form:
\begin{align}
    \partial_t \mathbf{w}(x, t)+\Lambda_{i} \partial_x \mathbf{w}(x, t)= \Theta_{i} \mathbf{w}(x,t), \label{eqwRiemann}
\end{align}
for $x \in (0,L)$, with the boundary condition:
\begin{align}
    \left[\begin{array}{l}
\mathbf{w}^{+}(0,t) \\
\mathbf{w}^{-}(1, t)
\end{array}\right]=G_{i}\left[\begin{array}{l}
\mathbf{w}^{+}(1,t) \\
\mathbf{w}^{-}(0, t)
\end{array}\right] + \left[\begin{array}{c}
     0  \\
    {U}(t)
\end{array}\right], \label{eqwRiemannbound}
\end{align}
where the coefficient matrix $\Theta_i$, $G_i$ are:
\begin{align*}
    \Theta_{i} =  \left[\begin{array}{cc}
    \Sigma^{++}_i(x) & \Sigma^{+-}_i(x) \\
    \Sigma^{-+}_i(x) & 0
\end{array}\right], G_{i} = \left[ \begin{array}{cc}
        0 & Q_i \\
        R_i & 0
    \end{array} \right]. 
\end{align*}
In the case of deterministic coefficients, equations~\eqref{eqwRiemann}-\eqref{eqwRiemannbound} naturally appear when modeling traffic networks with two classes of vehicles~\cite{burkhardt2021stop}.
\section{Backstepping Controller Design}
In this section, we propose a backstepping control design for the nominal system (that is, the system is in a known nominal deterministic state). We will then show that the stochastic system with this nominal control law is well-posed. Finally, we will state our main result, which is the mean-square exponential stability of the closed-loop system, provided the nominal parameters are sufficiently close to the stochastic ones on average. This result will be proved in the next section using a Lyapunov analysis.

\subsection{Backstepping transformation}

We consider in this section that the stochastic parameters are in the nominal mode $\mathscr{X}(t) = \mathscr{X}_{0}$. This nominal mode is not necessarily related to the set $\mathscr{S}$. Our objective is to design a control law that stabilizes this nominal system. We first simplify the structure of the system~\eqref{eqwRiemann} by removing the in-domain coupling terms in the equation. More precisely, let us consider the following backstepping transformation $\mathcal{K}_0$ 
\small
\begin{align}
    \mathcal{K}_0\mathbf{w}= \begin{pmatrix} \mathbf{w}^{+} \\\mathbf{w}^- -\int_0^x \mathbf{K}_0(x, \xi)\mathbf{w}^{+}(\xi,t) +N_0(x, \xi) \mathbf{w}^-(\xi, t)) d \xi \end{pmatrix}\label{back}
\end{align}
\normalsize
where the kernels $\mathbf{K}_0(x,\xi) \in \mathbb{R}^{1\times 3}$ and $N_0(x,\xi)\in\mathbb{R}^1$ are piecewise continuous functions defined on the triangular domain $\mathcal{T}=\{0 \leq \xi \leq x \leq 1\}$. We have
\begin{align}
    \mathbf{K}_0(x,\xi) = \left[ \begin{array}{ccc}
        k^0_1(x,\xi) & k^0_2(x,\xi) & k^0_3(x,\xi)
    \end{array}\right].
\end{align}
The different kernels verify the following PDE system
\small
\begin{align}
&-\Lambda^{-}_0(\mathbf{K}_0)_x(x,\xi) +(\mathbf{K}_0)_\xi(x,\xi)\Lambda^{+}_0=-\mathbf{K}_0(x,\xi)\Sigma^{++}_0(\xi) \notag\\
&-\Sigma^{-+}_0(\xi) N_0(x, \xi), \label{ker1} \\
&\Lambda^{-}_0 (N_0)_{x}(x, \xi)+\Lambda^{-}_0 (N_0)_{\xi}(x, \xi) =\mathbf{K}_0(x,\xi)\Sigma^{+-}_0(\xi), \label{ker2}
\end{align}
\normalsize
with the boundary conditions 
\begin{align}
\left(-\Lambda^{-}_0 \mathbf{I}_3-\Lambda^{+}_0\right) \mathbf{K}_0^\mathsf{T}(x,x)&=\Sigma^{-+^{\mathsf{T}}}_0(x) ,\label{ker3}\\
-\Lambda^{-}_0 N_0(x, 0)+  \mathbf{K}_0(x,0)\Lambda^{+}_0{Q}_0 &=0, \label{ker4}
\end{align}
where $\mathbf{I}_3$ is a $3\times 3$ identity matrix. The well-posedness of the kernel equations can be proved by adjusting the results from~\cite[Theorem 3.3]{hu2015control}. The solutions of the kernel equations can be expressed by integration along the characteristic lines. Applying the method of successive approximations, we can then prove the existence and uniqueness of the solution to the kernel equations~\eqref{ker1}-\eqref{ker4}.
Applying the backstepping transformation, we can define the target system state $\vartheta$ as $$\vartheta=(\alpha_1,\alpha_2,\alpha_3,\beta)=\mathcal{K}_0 \mathbf{w}.$$
And we denote the augmented states~$\mathscr{Q} = [\alpha_1,\alpha_2,\alpha_3]^\mathsf{T}$
The target system equations are given by:
\begin{small}
\begin{align}
&\mathscr{Q}_t(x,t)+\Lambda^{+}_0 \mathscr{Q}_x(x,t) =\Sigma^{++}_0(x)\mathscr{Q}(x,t)+\Sigma^{+-}_0(x) \beta(x,t) \nonumber\\
&+\int_0^x\mathbf{C}_0^+(x,\xi)\mathscr{Q}(\xi,t) d\xi +\int_0^x\mathbf{C}_0^-(x,\xi)\beta(\xi,t)d\xi,
\label{tar1}\\
&\beta_t\left(x, t\right)-\Lambda^{-}_0 \beta_x(x, t)=0, \label{tar2}
\end{align}
\end{small}
with the boundary conditions:
\begin{align}
\mathscr{Q}(0,t) =&{Q}_0 \beta(0, t), \label{tar3}\\
\beta(1, t)  =& R_0 \mathbf{w}^+(1,t) - \int_0^1 \mathbf{K}_0(1, \xi)\mathbf{w}^{+}(\xi,t)\nonumber\\
&+N_0(1, \xi) \mathbf{w}^-(\xi, t)) d \xi + \Bar{U}(t). \label{tar4}
\end{align}
where the coefficients $\mathbf{C}_0^+(x,\xi) \in \mathbb{R}^{3\times 3}$ and $\mathbf{C}_0^-(x,\xi) \in \mathbb{R}^{3\times 1}$ are bounded functions defined on the triangular domain $\mathcal{T}$. Their expressions can be found in~\cite{burkhardt2021stop}. Note that we still have the presence of $\mathbf{w}^+$ and $\mathbf{w}^-$ terms in equation~\eqref{tar4}, but this is not a problem since these terms will be removed using the control input. The transformation $\mathcal{K}_0$ is a Volterra transformation, therefore boundedly invertible. Consequently, the states $\mathbf{w}$ and $\vartheta$ have equivalent $L^2$ norms, i.e. there exist two constants $m_\vartheta>0$ and $M_\vartheta>0$ such that
 \begin{align}
    m_\vartheta||\mathbf{w}||_{L^2}^2 &\leq ||\vartheta||_{L^2}^2\leq M_\vartheta||\mathbf{w}||_{L^2}^2.
\end{align}
\subsection{Nominal control law and Lyapunov functional}
From the nominal target system~\eqref{tar1}-\eqref{tar4}, we can easily design a stabilizing control law as~\cite{auriol2016minimum}:
\begin{align}
    {U}(t)=&-{R}_0 \mathbf{w}^+(1,t) +\int_0^1\left(\mathbf{K}_0(1, \xi)\mathbf{w}^+(\xi,t) \right. \notag\\
    & \left. +N_0(1, \xi) \mathbf{w}^- (\xi, t)\right) d \xi. \label{controllaw}
\end{align}
To analyze the stability properties of the target system \eqref{tar1}-\eqref{tar4}, we consider the Lyapunov functional $V_0$ defined by 
\begin{align}
     V_0(t) = \int_0^1 \vartheta^\mathsf{T}(x,t) D_0(x) \vartheta(x,t)dx, \label{lyapunov0}
\end{align}
where 
\begin{align}
     D_0(x) = \text{Diag}\left\{\frac{\mathrm{e}^{-\frac{\nu}{\lambda_{1}^0}x}}{\lambda_{1}^0}, \frac{\mathrm{e}^{-\frac{\nu}{\lambda_{2}^0}x}}{\lambda_{2}^0}, \frac{\mathrm{e}^{-\frac{\nu}{\lambda_{3}^0}x}}{\lambda_{3}^0}, a \frac{\mathrm{e}^{\frac{\nu}{\Lambda^-_0} x}}{\Lambda^-_0} \right\}.
\end{align}
This Lyapunov functional is equivalent to the $L^2$ norm of the system, that is, there exist two constant  $k_1 > 0$ and $k_2>0$ such that
\begin{align}
    k_1||\vartheta||_{L^2}^2 \leq V_0(t) \leq k_2||\vartheta||_{L^2}^2. \label{eqnormV0}
\end{align}
It can also be expressed in terms of the original state as 
\begin{align}
V_0(t)=\int_0^1&(\mathcal{K}_0\mathbf{w}(x,t))^\mathsf{T} D_0(x)\mathcal{K}_0\mathbf{w}(x,t)dx.
\end{align}
Taking the time derivative of $V_0(t)$ and integrating by parts, we get
\begin{align}
    &\Dot{V}_0(t) \leq -\nu V_0(t) + \int_0^1 2  \mathscr{Q}(x,t) D_\alpha^0 \notag\\
    &\left(\Sigma_{0}^{++}(x)\mathscr{Q}(x,t) + \Sigma_{0}^{+-}(x) \mathbf{w}^-(x,t) \right)dx \notag \\
    &\leq - \eta V_0(t) +(q_{10}^2+q_{20}^2+q_{30}^2 - a) \beta^2(0,t),
\end{align}
where 
\begin{align}
    &\eta = \nu - \frac{2}{||\underline{\Lambda^+}|| k_1} (\max_{x\in [0,1]} ||\Sigma^{++}_0(x)||  \nonumber \\
    &+ (1 + \frac{1}{m_\vartheta}) \max_{x\in [0,1]} ||\Sigma^{+-}_0||(x)),\\
    &D_\alpha^0= \text{Diag}\left\{\frac{\mathrm{e}^{-\frac{\nu}{\lambda_{1}^0}x}}{\lambda_{1}^0}, \frac{\mathrm{e}^{-\frac{\nu}{\lambda_{2}^0}x}}{\lambda_{2}^0}, \frac{\mathrm{e}^{-\frac{\nu}{\lambda_{3}^0}x}}{\lambda_{3}^0} \right\}.
\end{align}
We choose $a>0$ and $\nu>0$ such that
\begin{align}
     & q_{10}^2+q_{20}^2+q_{30}^2 - a \leq 0, \eta>0.
\end{align}
where $q_{10}$, $q_{20}$, $q_{30}$ are the elements of $Q_0$. Consequently, we obtain $\dot{V}_0(t)\leq -\eta V_0(t)$, which implies the $L^2$-exponential stability of the system.
\subsection{Mean-square exponential stabilization}
We now state the well-posedness of the stochastic system and then give the main result on mean-square exponential stability. We must first guarantee that the stochastic system \eqref{clpsys1}-\eqref{clpsys4} with the nominal controller~\eqref{controllaw} has a unique solution. We have the following lemma,
\begin{lem}
    For any initial conditions of the Markov system $\mathbf{w}(x,t) \in L^2[0,1]$ and any initial states $\mathscr{X}(t) = \mathscr{X}(0)$ for the stochastic parameters, the system \eqref{clpsys1}-\eqref{clpsys4} with the nominal control law \eqref{controllaw} has a unique solution such that for any $t$,
\begin{align}
    \mathbb{E}\{ ||\mathbf{w}(x,t)|| \} < \infty,\label{unique}
\end{align}
where the $\mathbb{E}\{\cdot\}$ denotes the mathematical expectation.
\end{lem}
\begin{proof}
    This lemma can be easily proved by adjusting the results in \cite{zhang2017stochastic}. Almost every sample path of our stochastic processes are right-continuous step functions with a finite number of jumps in any finite time interval. We can then find a sequence $\{ t_k: k =0,1,\dots \}$ of stopping times such that $t_0 = 0, \lim_{t\to \infty} t_k = \infty$, and $\mathscr{X}(t)=\mathscr{X}(t_k)$ on $t_k \leq t < t_{k+1}$. We start from time $t=0$ and then use \cite[Theorem A.4]{bastin2016stability} for each time interval in the whole time period. Thus, the stochastic system~\eqref{clpsys1}-\eqref{clpsys4} has a unique solution.
\end{proof}
The main goal of this paper is to prove that the control law \eqref{controllaw} can still stabilize the stochastic system \eqref{clpsys1}-\eqref{clpsys4}, provided the nominal parameter $\mathscr{X}_0$ is sufficiently close to the stochastic ones on average. More precisely, we want to show the following sufficient condition for robust stabilization. 
\begin{thm}
There exists a constant $\epsilon^\star>0$, such that if, for all time $t>0$,
\begin{align}
  \mathbb{E}\left(\left|\left|\mathscr{X}(t) - \mathscr{X}_0\right|\right|\right) \leq \epsilon^\star, \label{eqinegstocha}
\end{align}
then the closed-loop system \eqref{clpsys1}-\eqref{clpsys4} with the control law \eqref{controllaw} is mean-square exponentially stable, namely, there exist $\varsigma,\zeta>0$ such that:
\begin{align}
    \mathbb{E}_{[0,(p(0),\mathscr{X}(0)]}(p(t)) \leq \varsigma \mathrm{e}^{-\zeta t} p(0),
\end{align}
where $p(t)=\int_0^1 ||\mathbf{w}(x,t)||_2^2  dx$, while $\mathbb{E}_{[0,(p(0),\mathscr{X}(0)]}$ denotes the conditional expectation at time $t=0$ with initial settings of $p(t) = p(0)$, $\mathscr{X}(t) = \mathscr{X}(0)$.\label{mainthm}
\end{thm}
This theorem will be proved in the next section.
\section{Lyapunov Analysis}
In this section, we consider the closed-loop stochastic system with the nominal controller~\eqref{controllaw}. The objective is to prove Theorem 2. The proof will rely on a Lyapunov analysis. More precisely, we will consider the following stochastic Lyapunov functional candidate
\begin{align}
    V(t)=\int_0^1&(\mathcal{K}_0\mathbf{w}(x,t))^\mathsf{T} D(t,x)\mathcal{K}_0\mathbf{w}(x,t)dx, \label{Lyapfun}
\end{align}
where the diagonal matrix $D(t,x)=D_j(x)$ if $\mathscr{X}(t)=\mathscr{X}_j$, and where
\begin{align}
     D_j(x) = \text{Diag}\left\{\frac{\mathrm{e}^{-\frac{\nu}{\lambda_{1}^j}x}}{\lambda_{1}^j}, \frac{\mathrm{e}^{-\frac{\nu}{\lambda_{2}^j}x}}{\lambda_{2}^j}, \frac{\mathrm{e}^{-\frac{\nu}{\lambda_{3}^j}x}}{\lambda_{3}^j}, a \frac{\mathrm{e}^{\frac{\nu}{\Lambda^-_j} x}}{\Lambda^-_j} \right\}.
    \label{defDj}
\end{align}
We consider that the parameters $\nu$ and $a$ introduced in the definition of $D_j$ can still be tuned. In the nominal case $\mathscr{X}(t) = \mathscr{X}_0$, the Lyapunov functional $V(t)$ corresponds to $V_0$. It is noted that inequality~\eqref{eqnormV0} still holds for $V(t)$ (even if the constants $k_1$ and $k_2$ may change). 
\subsection{Target system in stochastic mode $\mathscr{X}_j$}
In this section, we consider that~$\mathscr{X}(t)=\mathscr{X}_j$ at time $t$. We can define the state $\vartheta=(\alpha_1,\alpha_2,\alpha_3,\beta)=\mathcal{K}_0\mathbf{w}$. Our objective is first to obtain the equations verified by the state $\vartheta$ that appears in the Lyapunov functional~\eqref{Lyapfun}.  It verifies the following set of equations
\begin{align}
\mathscr{Q}_t(x,t)+\Lambda^{+}_j \mathscr{Q}_x(x,t) =& \Sigma^{++}_j(x)\mathbf{w}^+(x,t) \notag \\
&+\Sigma^{+-}_j(x) \mathbf{w}^-(x,t), \label{tar1sto}
\end{align}
\begin{align}
&\beta_t\left(x, t\right)-\Lambda^{-}_j \beta_x(x, t) =\mathbf{f}_{1j}(x) \mathbf{w}^+(x,t)
+\mathbf{f}_{2j}(x) \beta(0, t) \notag \\
&+\int_0^x \mathbf{f}_{3j}(x,\xi)\mathbf{w}^+(\xi,t) d \xi 
+\int_0^x \mathbf{f}_{4j}(x,\xi) \mathbf{w}^-(\xi,t) d \xi, \label{tar2sto}
\end{align}
with the boundary conditions:
\begin{align}
\mathscr{Q}(0,t) &={Q}_j \beta(0, t), \label{tar3sto}\\
\beta(1, t) &=\left({R}_j-{R}_0\right)\mathscr{Q}(1,t), \label{tar4sto}
\end{align}
where the functions are defined by:
\begin{align}
    \mathbf{f}_{1j}(x) &= \Sigma^{-+}_j(x)+\Lambda^-_j \mathbf{K}_0(x, x)+\mathbf{K}_0(x, x) \Lambda^{+}_j,\\
    \mathbf{f}_{2j}(x) &= -\mathbf{K}_0(x, 0) \Lambda^{+}_j {Q}_j+N_0(x, 0) \Lambda^{-}_j,\\
    \mathbf{f}_{3j}(x,\xi) &= \Lambda^{-}_j (\mathbf{K}_0)_{x}(x, \xi)-(\mathbf{K}_0)_{\xi}(x, \xi) \Lambda^{+}_j \notag \\
     &- \mathbf{K}_0(x, \xi) \Sigma^{++}_j(\xi)-N_0(x, \xi) \Sigma^{-+}_j(\xi),\\
    \mathbf{f}_{4j}(x,\xi) &= \Lambda^{-}_j (N_0)_{x} (x, \xi)+\Lambda^{-}_j (N_0)_{\xi}(x, \xi) \notag\\
    &-\mathbf{K}_0(x, \xi) \Sigma^{+-}_j(\xi).
\end{align}
All the terms that depend on $\mathbf{w}$ in the target system~\eqref{tar1sto}-\eqref{tar4sto} could be expressed in terms of $\vartheta$ using the inverse transformation~$\mathcal{K}_0^{-1}$. However, this would make the computations more complex and is not required for the stability analysis.
It is important to emphasize that all the terms on the right-hand side of equation \eqref{tar2sto} become \emph{small} if the stochastic parameters are close enough to the nominal ones. More precisely, we have the following lemma
\begin{lem}
There exists a constant $M_0$, such that for any realization $\mathscr{X}(t)=\mathscr{X}_j \in \mathscr{S}$, for any $(x,\xi)\in \mathcal{T}$
    \begin{align}
        ||\mathbf{f}_{\mathfrak{i}j}|| < M_0  \left|\left|\mathscr{X}_j-\mathscr{X}_0\right|\right|, \quad \mathfrak{i} \in \{1,2,3,4\}.
    \end{align}\label{lemboundf}
\end{lem}
\begin{proof}
    Considering the function $\mathbf{f}_{1j}(x)$. For all $x \in [0,1]$, we have
\begin{align}
    \mathbf{f}_{1j}(x) &= \Sigma^{-+}_j(x)+\Lambda^-_j \mathbf{K}_0(x, x)+\mathbf{K}_0(x, x) \Lambda^{+}_j\notag \\
    &= (\Sigma^{-+}_j(x) - \Sigma^{-+}_0(x)) + (\Lambda^-_j-\Lambda^-_0) \mathbf{K}_0(x,x) \notag\\
    &+ \mathbf{K}_0(x,x)(\Lambda^{+}_j-\Lambda^+_0).
\end{align}
Consequently, we obtain the existence of a constant $K_1>0$ such that
\begin{align}
    ||\mathbf{f}_{1j}|| &\leq K_1\left|\left|\mathscr{X}_j-\mathscr{X}_0\right|\right|.
\end{align}
The other inequalities for $\mathbf{f}_2(x)$, $\mathbf{f}_3(x,\xi)$ and $\mathbf{f}_4(x,\xi)$ can also be derived similarly. This finishes the proof.
\end{proof}
\subsection{Derivation of the Lyapunov function}
Let us consider the Lyapunov functional~ $V$ defined in equation~\eqref{Lyapfun}. Its infinitesimal generator $L$ is
defined as~\cite{ross2014introduction}
\begin{align}
 L V(\mathbf{w},s_2) &=\limsup _{\Delta t \rightarrow 0^{+}} \frac{1}{\Delta t} \times \mathbb{E}(V(\mathbf{w}(t+\Delta t), \mathscr{X}(t+\Delta t))\notag\\
&-V(\mathbf{w}(t), \mathscr{X}(t))).
\end{align}
We define $L_j$, the infinitesimal generator of $V$ obtained  by fixing $\mathscr{X}(t) = \mathscr{X}_j \in \mathscr{S}$. We have\begin{align}
L_j V(\mathbf{w}) & =\frac{d V}{d \mathbf{w}}\left(\vartheta, \mathscr{X}_j\right) h_j(\vartheta) +\sum_{\ell \in \mathscr{S}}\left(V_{\ell}(\mathbf{w})-V_j(\mathbf{w})\right) \tau_{j \ell},
\end{align}
where $V_\ell(\mathbf{w})=V(\mathbf{w},s_2^\ell)$, and where the operator $h_j$ is defined by 
\begin{align}
    h_j(\vartheta)=\left(\begin{array}{l}
-\Lambda^+_{j} \mathscr{Q}_x(x,t)+\Sigma_{j}^{++}(x)\mathbf{w}^+(x,t)\\
+\Sigma_{j}^{+-}(x)\mathbf{w}^-(x,t)\\
\Lambda_{j}^{-} \beta_x(x, t)+\mathbf{f}_{1j}(x)\mathbf{w}^+(x,t) \\
+ \mathbf{f}_{2j}(x) \beta(0, t)+\int_0^x \mathbf{f}_{3j}(x,\xi)\mathbf{w}^+(\xi,t) d \xi\\
+\int_0^x \mathbf{f}_{4j}(x,\xi)\mathbf{w}^-(\xi,t) d \xi
\end{array}\right).
\end{align}
To shorten the computations, we denote in the sequel $V(t), LV(t), V_j(t)$ and $L_jV(t)$ instead of (respectively) $V(\mathbf{w},\mathscr{X}(t))$, $LV(\mathbf{w},\mathscr{X}(t))$, $V(\mathbf{w},\mathscr{X}_j)$ and $L_j(V(\mathbf{w}))$. From now, we consider that $\mathscr{X}(t=0)=\mathscr{X}_{i_0}\in\mathscr{S}$. We have the following lemma.
\begin{lem}
There exists $\Bar{\eta}>0$, $M_1 > 0$ and $d_1,d_2 > 0$ such that the Lyapunov functional $V(t)$ satisfies
\begin{align}
&\sum_{j=1}^r P_{i j}(0, t) L_j V(t) \leq -V(t)\Bigg(\Bar{\eta}-d_1 \mathcal{Z}(t) \notag\\
&\left.-\left(M_1+d_1 r \tau^{\star}\right) \mathbb{E}\left(\left|\left|\mathscr{X}(t)-\mathscr{X}_0\right|\right|\right)\right)\notag \\
&+\sum_{k=1}^3 (d_2 \mathbb{E}(\left|\left|\mathscr{X}(t)-\mathscr{X}_0\right|\right|)-\mathrm{e}^{-\frac{\nu}{\bar \lambda}})\alpha_k^2(1,t)
\end{align}
where the function $\mathcal{Z}(t)$ is defined as:
\begin{align}
\mathcal{Z}(t)=\sum_{j=1}^r\left|\left|\mathscr{X}(t)-\mathscr{X}_0\right|\right|\left(\partial_t P_{i j}(0, t)+\mathfrak{c}_j P_{i j}(0, t)\right)
\end{align} \label{lemlyapunovfunctional}
\end{lem}
\begin{proof}
    In what follows, we denote $c_i$ positive constants. We will first compute the first term of $L_j$. Consider that $\mathscr{X}(t) =\mathscr{X}_j $. The Lyapunov functional rewrites
    \begin{align}
        V_j(t)=\int_0^1\vartheta^\mathsf{T}(x,t)D_j(x) \vartheta(x,t)dx,
    \end{align}
    \begin{align}
&\frac{d V_j}{d \mathbf{w}}(\mathbf{w}) h_j(\mathbf{w})  \leq-\eta V_j(t)+M_1 \left|\left|\mathscr{X}_j-\mathscr{X}_0\right|\right| V(t) 
\notag \\
&+\left(c_2\left|\overline{\mathscr{X}}-\underline{\mathscr{X}}\right|\varepsilon_0+q_{1j}^2+q_{2j}^2+q_{3j}^2-a\right) \beta^2(0, t) \notag \\
&+\sum_{k=1}^3 (a \mathrm{e}^{\frac{\nu}{\Lambda^-_j}}\left(({R}_j)_k-({R}_0)_k\right)^2 -\mathrm{e}^{-\frac{\nu}{\lambda_{kj}}L})\alpha_k^2(1,t),
\end{align}
where 
\begin{align}
&\eta = \nu - \frac{2}{||\underline{\Lambda^+}|| k_1} (\max_{x\in [0,1]} ||\Sigma^{++}_0(x)||  \nonumber \\
    &+ (1 + \frac{1}{m_\vartheta}) \max_{x\in [0,1]} ||\Sigma^{+-}_0(x)||)-\frac{2 \overline{\mathscr{X}} c_2}{k_1\varepsilon_0}, \\
    &M_1=c_4+ac_3+\frac{c_2}{k_1\varepsilon_0}+c_1.
\end{align}
The coefficients $\nu$, $\varepsilon_0$ and $a$ are chosen such that
\begin{align}
    &\eta >0, \quad c_2\left|\overline{\mathscr{X}}-\underline{\mathscr{X}}\right|\varepsilon_0+{q}_{1j}^2+{q}_{2j}^2+{q}_{3j}^2-a <0.
\end{align}
where the ${q}_{1j}$, ${q}_{2j}$, ${q}_{3j}$ are the elements of ${Q}_j$, $\overline{\mathscr{X}}$ and $\underline{\mathscr{X}}$ are the upper and lower bound of the stochastic parameters. There exists a constant $C_0$ such that for all $1\leq j \leq r$, $V_j(\mathbf{w}) \leq C_0 V(\mathbf{w})$.
Thus, we get the following inequality:
\begin{align}
    &\frac{d V_j}{d \mathbf{w}}(\mathbf{w}) h_j(\mathbf{w})  \leq-\bar \eta V(t)+M_1 \left|\left|\mathscr{X}_j-\mathscr{X}_0\right|\right| V(t) \nonumber \\
    &+\sum_{k=1}^3 (a \mathrm{e}^{\frac{\nu}{\Lambda^-_j}}\left(({R}_j)_k-({R}_0)_k\right)^2 -\mathrm{e}^{-\frac{\nu}{\lambda_{kj}}L})\alpha_k^2(1,t),\label{ineq1}
\end{align}
where $\bar \eta=\eta C_0$. Now, we calculate the second term of $L_j$. We have:
\begin{align}
    &\sum_{l=1}^r\left(V_l(\mathbf{w})-V_j(\mathbf{w})\right) \tau_{j l} =\sum_{l=1}^r  \tau_{j l}\notag \\
    &\left( \int_0^1 \mathcal{K}_0^\mathsf{T}(\mathbf{w}(x,t)) D_l(x) \mathcal{K}_0\mathbf{w}(x,t) dx  \right.\nonumber\\
    &- \left.  \int_0^1 \mathcal{K}_0^\mathsf{T}\mathbf{w}(x,t) D_j(x) \mathcal{K}_0\mathbf{w}(x,t) dx  \right).\nonumber\\
    &\leq  d_1 \sum_{l=1}^r \tau_{jl}\left|\left|\mathscr{X}_l - \mathscr{X}_j\right|\right| V(t), \label{ineq2}
\end{align}
We then calculate the quantity $\Bar{L} = \sum_{j=1}^r P_{ij}(0,t) L_jV(t)$. Using the property of the expectation and we get
\begin{align}
    &\Bar{L} \leq -V(t)\left(\bar \eta - (M_1+d_1 r \tau^\star) \mathbb{E}(\left|\left|\mathscr{X}(t)-\mathscr{X}_0\right|\right|) \right. \notag\\
    & \left.+ d_1 \sum_{j=1}^r\left|\left|\mathscr{X}_j-\mathscr{X}_0\right|\right|\left(\partial_t P_{i j}(0, t)+\mathfrak{c}_j P_{i j}(0, t)\right)\right) \notag \\
    &+\sum_{k=1}^3 (d_2 \mathbb{E}(\left|\left|\mathscr{X}(t)-\mathscr{X}_0\right|\right|)-\mathrm{e}^{-\frac{\nu}{\bar \lambda}})\alpha_k^2(1,t),\end{align}
    This finish the proof of Lemma~\ref{lemlyapunovfunctional}.
\end{proof} 
\subsection{Proof of Theorem 2}
Notice first that if $\epsilon^\star$ is small enough (namely smaller than $\frac{e^{-\frac{\nu}{\bar \lambda}}}{d_2}$) and if inequality~\eqref{eqinegstocha} holds, the term $\sum_{k=1}^3 (d_2 \mathbb{E}(\left|\left|\mathscr{X}(t)-\mathscr{X}_0\right|\right|)-\mathrm{e}^{-\frac{\nu}{\bar \lambda}})\alpha_k^2(1,t) < 0$, then we have the following result based on Lemma \ref{lemlyapunovfunctional}: 
\begin{align}
&\sum_{j=1}^r P_{i j}(0, t) L_j V(t) \leq -V(t)\Bigg(\Bar{\eta}-d_1 \mathcal{Z}(t) \notag\\
&\left.-\left(M_1+d_1 r \tau^{\star}\right) \mathbb{E}\left(\left|\left|\mathscr{X}(t)-\mathscr{X}_0\right|\right|\right)\right).
\end{align}
We define the following function:
\begin{align}
    \phi(t) = \Bar{\eta }- d_1 \mathcal{Z}(t) - (M_1 + d_1 r \tau^\star)\mathbb{E}(||\mathscr{X}(t)-\mathscr{X}_0||).
\end{align}
And then, using the functional $\Psi(t)$:
\begin{align}
    \Psi(t) = \mathrm{e}^{\int_0^t \phi(y) dy} V(t).
\end{align}
With the definition of $\Psi(t)$, taking the expectation of the infinitesimal generator $L$ of $\Psi(t)$ , we get:
\begin{align}
    \mathbb{E}\left(\sum_{j=1}^r P_{ij}(0,t) L_{j}V(t)\right)\leq - \mathbb{E}\left(V(t)\phi(t)\right).
\end{align}
We know that $\mathbb{E}(\sum_{j=1}^r P_{ij}(0,t) L_{j}V(t)) =  \mathbb{E}(LV(t))$, thus
\begin{align}
    \mathbb{E}(LV(t))  \leq - \mathbb{E}(V(t)\phi(t)). \label{inequ}
\end{align}
Then applying the Dynkin's formula~\cite{dynkin2012theory},
\begin{align}
    \mathbb{E}(\Psi(t)) - \Psi(0) = \mathbb{E}\left(\int_0^t L\Psi(y)dy\right) \leq 0.
\end{align}
To calculate the $\mathbb{E}(\Psi(t))$, we write down the formulation of $\Psi(t)$:
\begin{align}
    &\mathbb{E}(\Psi(t)) =  \mathbb{E}\left( V(t) \mathrm{e}^{\int_0^t \phi(y)dy} \right) \notag \\
    & = \mathbb{E}\left( V(t) \mathrm{e}^{\int_0^t (\Bar{\eta} - d_1 \mathcal{Z}(y) - (M_1 + d_1 r \tau^*) \mathbb{E}(||\mathscr{X}(y)-\mathscr{X}_0||))dy} \right).
\end{align}
We already know that 
\begin{align}
    &\int_0^t \mathcal{Z}(y)dy =  \int_0^t \left(\sum_{j=1}^r\left|\left|\mathscr{X}(y)-\mathscr{X}_0\right|\right|\left(\partial_y P_{i j}(0, y) \right.\right.\notag \\
    &\left. +\mathfrak{c}_j P_{i j}(0, y)\right) V(y)\Big) dy \notag\\
    &\leq \mathbb{E}(\left|\left|\mathscr{X}(t)-\mathscr{X}_0\right|\right|) + r\tau^\star \int_0^t \mathbb{E} (\left|\left|\mathscr{X}(y)-\mathscr{X}_0\right|\right|) dy,
\end{align}
where $\tau^\star$ is the largest value of the transition rate. Using this inequality, we get
\begin{align}
    \mathbb{E}(\Psi(t))\geq \mathbb{E} \left( V(t) \mathrm{e}^{\left(-d_1 \epsilon^\star + \int_0^t (\Bar{\eta} - (M_1 + 2d_1 r \tau^\star) \epsilon^\star dy\right)} \right).
\end{align}
Then we take $\epsilon^\star$ as 
\begin{align}
    \epsilon^\star = \frac{\Bar{\eta}}{2(2d_1 r \tau^\star + M_1)},
\end{align}
thus we have 
\begin{align}
    \mathbb{E}(\Psi(t)) \geq \mathbb{E} \left( V(t) \mathrm{e}^{\left(-d_1 \epsilon^\star + \frac{\Bar{\eta}}{2} t\right)} \right).
\end{align}
From the before proof, we know $\mathbb{E}(\Psi(t)) \leq \Psi(0)$, such that
\begin{align}
    \mathbb{E}(V(t)) \leq  \mathrm{e}^{d_1 \epsilon^\star} \mathrm{e}^{-\zeta t} V(0),
\end{align}
where $\zeta = \frac{\Bar{\eta}}{2}$. The function $V(t)$ is equivalent to the $L^2$-norm of the system. This concludes the proof of Theorem \ref{mainthm}.

\section{Numerical Simulation}
In this section, we illustrate our results with simulations. We consider that only the parameter $\lambda_4$ is stochastic. Its
nominal value is $-0.024$. The five other possible values are $\lambda_4^1=-0.02, \lambda_4^2=-0.023, \lambda_4^3=-0.024, \lambda_4^4=-0.025, \lambda_4^5=-0.03)$ and the initial transition probabilities are chosen as $(0.02,0.32,0.32,0.32,0.02)$. The transition rates $\tau_{ij}$ are defined as the same as in ~\cite{auriol2023mean}.
The corresponding matrices in nominal case are setting as:
\begin{align*}
    &\Lambda^{+}_0 =\left[\begin{array}{ccc}
0.0081 & 0 & 0 \\
0 & 0.0037 & 0 \\
0 & 0 & 0.0065
\end{array}\right], \Lambda^{-}_0 =-0.024\\
&Q_0 = \begin{bmatrix}
        -12.29 \\
        -3\\
        8.45
\end{bmatrix}, R_0 = \begin{bmatrix}
        0.0011 & -0.1601 & 0.0034
\end{bmatrix}
\end{align*}

Solving the Kolmogorov forward equation, we get the probability of each state in the simulation process shown in Fig. \ref{Pro}.
\begin{figure}
    \centering
    \includegraphics[width=0.46\textwidth]{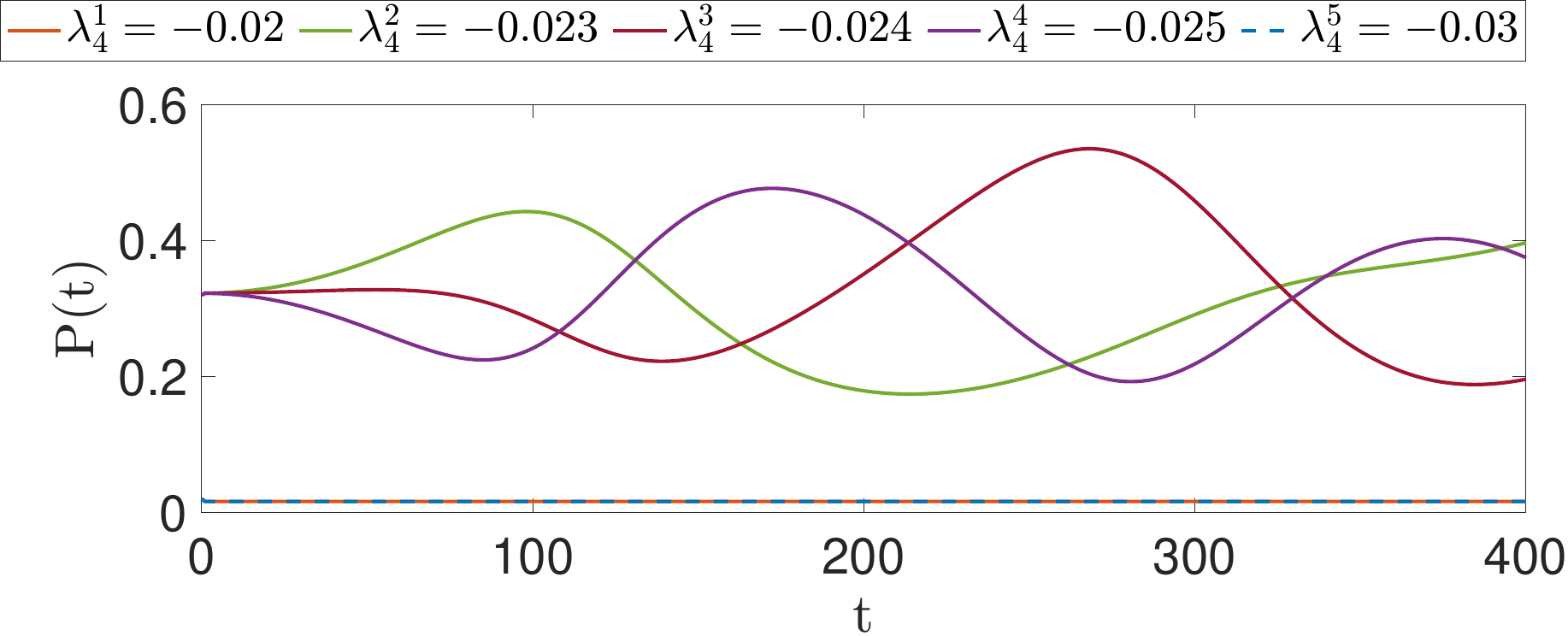}
    \caption{The probability of Markov states}
    \label{Pro}
\end{figure}
From the probability of the Markov states, the system stays near the nominal value in the entire simulation period. Using the Markov process, we conduct the simulation for $t=400$ with the sinusoidal initial conditions, the closed-loop results are shown in Fig. \ref{closedloopres}.
\begin{figure}
\centering
    \begin{subfigure}{0.23\textwidth}
        \includegraphics[width =\textwidth]{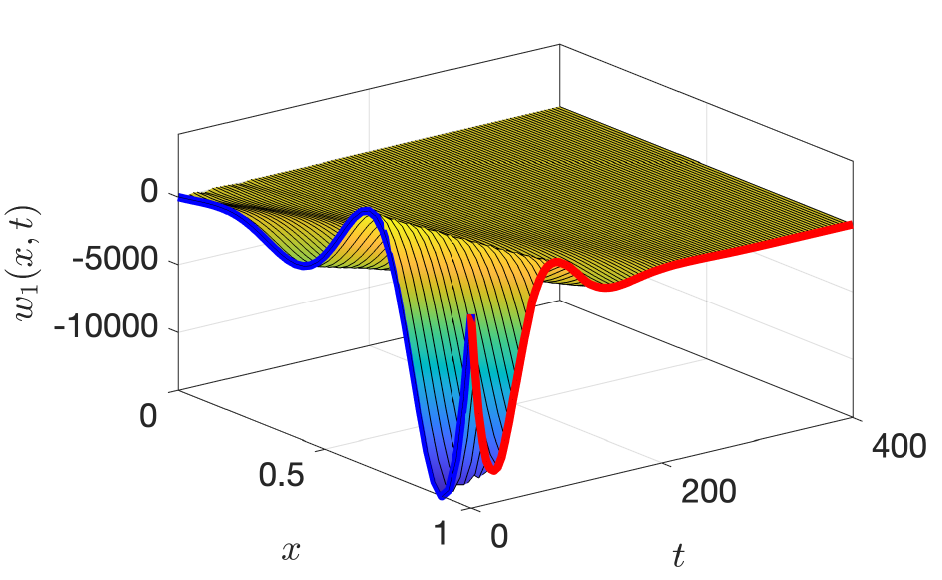}
        \caption{$w_1$}
        \label{w_1}
    \end{subfigure}
    \begin{subfigure}{0.23\textwidth}
        \includegraphics[width =\textwidth]{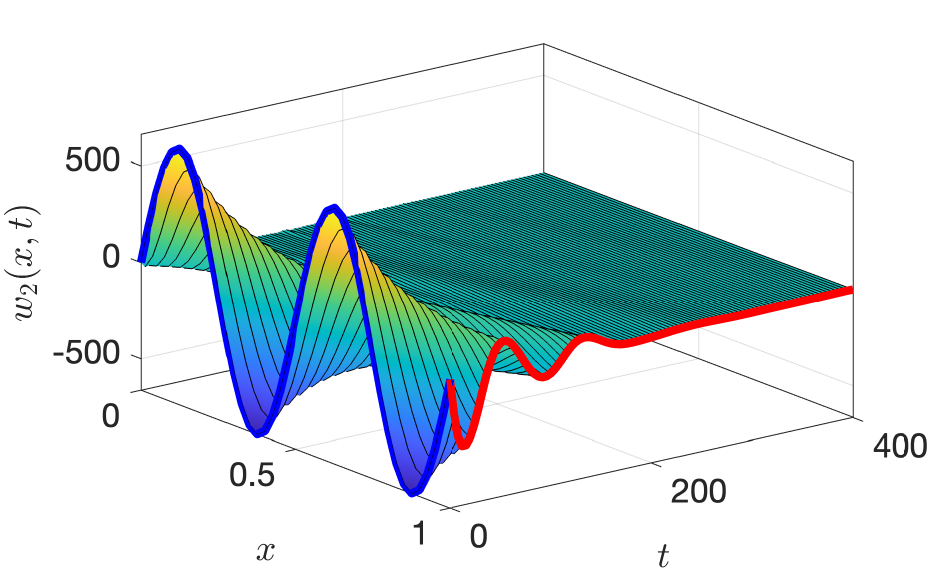}
        \caption{$w_2$}
        \label{w_2}
    \end{subfigure}
    \begin{subfigure}{0.23\textwidth}
        \includegraphics[width =\textwidth]{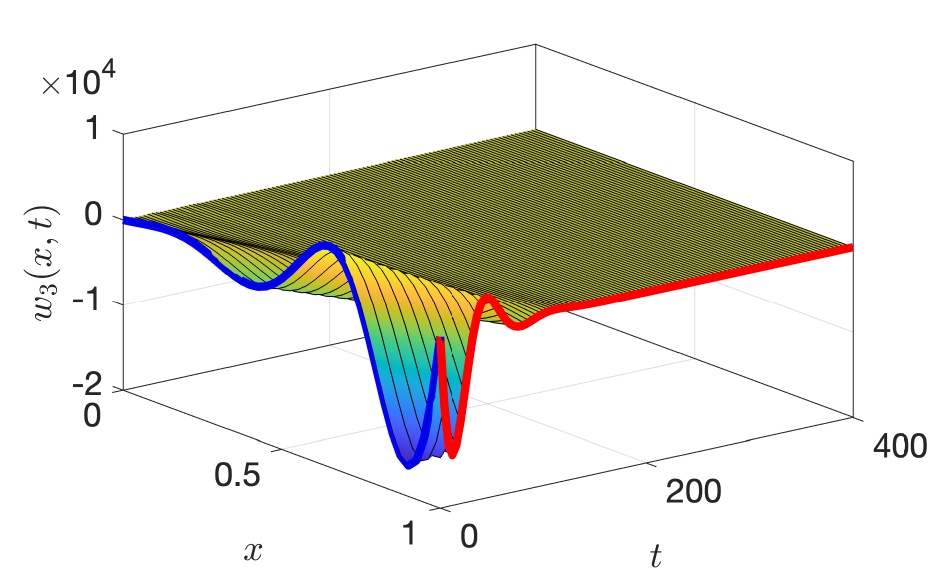}
        \caption{$w_3$}
        \label{w_3}
    \end{subfigure}    
    \begin{subfigure}{0.23\textwidth}
        \includegraphics[width =\textwidth]{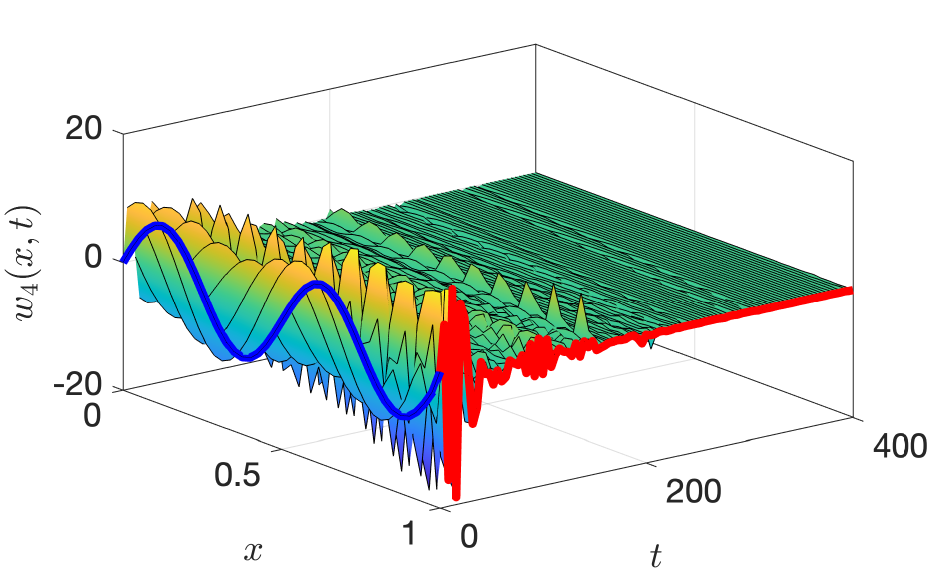}
        \caption{$w_4$}
        \label{w_4}
    \end{subfigure}
    \caption{Closed-loop results of the stochastic system}
    \label{closedloopres}
\end{figure}
All the states with Markov jumping parameters almost converge to zero under the nominal control law, which is consistent with the theoretical results. 
\section{Conclusions}
In this paper, we proposed a backstepping control low that mean-squarely exponentially stabilizes a $4\times 4$ Markov jumping  coupled hyperbolic PDEs. The full-state feedback boundary control law was derived using the backstepping method for a nominal system. By applying Lyapunov analysis, we prove that this nominal control law can stabilize the PDE system with Markov jumping parameters provided the nominal parameters are sufficiently close to the stochastic ones on average. Finally, we use numerical examples to illustrate the efficiency of our approach.Future work will focus on its application in traffic flow systems.

\addtolength{\textheight}{-12cm}   








\bibliography{ref}

\end{document}